\documentclass[12pt, a4paper, reqno]{amsart}

\usepackage[abbrev]{amsrefs}
\usepackage{amscd}
\usepackage[dvipdfmx]{hyperref}
\hypersetup{colorlinks=true,linkcolor=blue,citecolor=blue}
\usepackage{mathrsfs}
\usepackage{graphicx}

\allowdisplaybreaks[4]

\theoremstyle{plain}
  \newtheorem{thm}{Theorem}[section]
  \newtheorem{lem}[thm]{Lemma}
  
  \newtheorem{cor}[thm]{Corollary}
  \newtheorem{prop}[thm]{Proposition}

\theoremstyle{definition}
  \newtheorem{defn}[thm]{Definition}
  
  \newtheorem{ex}[thm]{Example}

\theoremstyle{remark}
  \newtheorem{rem}[thm]{Remark}

\renewcommand{\labelenumi}{(\theenumi)}
\numberwithin{equation}{section}

\DeclareMathOperator{\im}{im}
\DeclareMathOperator{\kernel}{ker}

\makeatletter

\@addtoreset{equation}{section}
\makeatother

\begin{document}
\title[The Ricci curvature on simplicial complexes]{The Ricci curvature on simplicial complexes}

\author[T. Yamada]{Taiki Yamada}
\address{Interdisciplinary Faculty of Science and Engineering, Shimane University, Shimane 690-8504, Japan}
\email{taiki\_yamada@riko.shimane-u.ac.jp}

\pagestyle{plain}

\footnote{
This work was supported by the Research Institute for Humanity and Nature (Director-General’s Discretionary Budge) and JSPS  KAKENHI(21K13800).\\
2010 \textit{Mathematics Subject Classification}.
Primary 05C50; Secondary 53C21 .\\
Keywords; Simplicial complex, Ollivier's Ricci curvature, Eigenvalue of Laplacian.
}

\maketitle
\begin{abstract}
We define the Ricci curvature on simplicial complexes by modifying the definition of the Ricci curvature on graphs, and we prove the upper and lower bounds of the Ricci curvature. These properties are generalizations of previous studies. Moreover, we obtain an estimate of the eigenvalues of the Laplacian on simplicial complexes using the Ricci curvature. 
\end{abstract}

\section{Introduction}
 The Ricci curvature plays an important role in Riemannian geometry when we investigate the global properties of a manifold. The Bonnet--Myers theorem and the Lichnerowicz theorem are good examples. Recently, Ollivier \cite{Ol1} generalized the Ricci curvature through a method now known as Ollivier's coarse Ricci curvature. By modifying this Ricci curvature, the Ricci curvature of a graph $G=(V, E)$ for two distinct vertices $x, y \in V$ is defined as
\begin{eqnarray*}
\kappa(x, y) = 1 - \cfrac{W(m_x, m_y)}{d(x, y)},
\end{eqnarray*}
where $m = \left\{m_x \right\}_{x \in V}$ is a random walk on $V$ and $W(m_x, m_y)$ is the $1$-Wasserstein distance between $m_x$ and $m_y$. In 2010, Lin, Lu, and Yau \cite{Yau1} defined the Ricci curvature of an undirected graph by using lazy random walk, and they studied the Ricci curvature of the product space of graphs and random graphs. In 2012, Jost and Liu \cite{Jo2} defined the Ricci curvature on undirected graphs by using a simple random walk; therefore, their definition is simpler than Lin-Lu-Yau's definition. In 2019, we first introduced the notion of the Ricci curvature of directed graphs \cite{Y1}. We proved some global properties of directed graphs, but we had to assume a slightly strong condition. To solve this problem, in 2020, Eidi-Jost \cite{Ei} defined the Ricci curvature on directed hypergraphs. Recently, we studied the transportation inequality along the heat flow and the concentration of measure inequality for directed graphs \cite{YSO1}. In contrast, the coarse Ricci curvature on graphs has been applied to many fields, such as cancer networks \cite{Tan} and Internet topology \cite{Ni}. In this paper, we proved the upper and lower bounds of the Ricci curvature on simplicial complex modification \cite{Jo2}.
\begin{thm}\label{lower}
For $i$-faces $F$ and $F'$ with $F \sim F'$, if the weights of the facets are equal to 1, then we have
\begin{eqnarray*}
\kappa (F, F') &\geq& - \cfrac{1}{i+1} \left( 1 - \cfrac{1}{\deg F} - \cfrac{1}{\deg F'} - \cfrac{i \# \Gamma (F, F')}{\deg F'} - \cfrac{\# \Gamma (F, F')}{\deg F \wedge \deg F'}  \right)_+ \\
&\ & -\cfrac{1}{i+1} \left( 1 - \cfrac{1}{\deg F} - \cfrac{1}{\deg F'} - \cfrac{i \# \Gamma (F, F')}{\deg F'} - \cfrac{\# \Gamma (F, F')}{\deg F \vee \deg F'}  \right)_+\\
&\ & +\cfrac{\# \Gamma (F, F')}{(i+1)(\deg F \vee \deg F')}
\end{eqnarray*}
where $\deg F \wedge \deg F = \min \left\{ \deg F, \deg F' \right\}$, $\deg F \vee \deg F = \max \left\{ \deg F, \deg F' \right\}$, and for $r \in \mathbb{R}$, $(r)_+ = r$ if $r>0$; otherwise, it is 0.
\end{thm}

The graph Laplacian is one of the most important concepts for obtaining the analytic properties of directed graphs. In 1847, Kirchhoff \cite{Kir} first defined the graph Laplacian on real-valued functions. Subsequently, many studies have investigated the graph Laplacian and its spectrum. There are several definitions of the graph Laplacian in addition to the Kirchhoff graph Laplacian. One of the most famous definitions is the normalized graph Laplacian defined by Bottema \cite{Bo}. The second is the {\em higher-order combinatorial Laplacian} defined by Eckmann \cite{Eck}. This Laplacian is defined on simplicial complexes, not graphs, and is based on a discrete version of the Hodge theory. In 2013, by modifying the higher-order combinatorial Laplacian, Jost and Horak \cite{Ho} developed a general framework for Laplacians defined by the combinatorial structure of a simplicial complex. This Laplacian is called the {\em $i$-Laplacian} (see Definition \ref{i-laplacian}). If we apply several weight functions to an inner product as given in Definition \ref{innerdef}, then the $i$-Laplacian with this weight covers every graph Laplacian obtained previously, so the $i$-Laplacian is the best definition of the graph Laplacian. Additionally, one of the most important properties of the $i$-Laplacian is that its nonzero spectrum on vertices coincides with that on edges (see (2.6) in \cite{Ho}). Hence, in this study, we used the $i$-Laplacian.

There is a good relationship between the coarse Ricci curvature and the Laplacian graph. In 2009, Ollivier \cite{Ol1} estimated the first non-zero eigenvalue of the normalized graph Laplacian by a lower bound of the coarse Ricci curvature.
\begin{thm}[Ollivier, Proposition 30 in \cite{Ol1}]
\label{Ol}
Let $\mu$ be the non-zero eigenvalue of the normalized graph Laplacian $\Delta_{0}$. Suppose that $\kappa(x, y) \geq k$ for any $(x, y) \in E$ and for a positive real number $k$. Then, we have
	\begin{eqnarray*}
	k \leq \mu \leq 2 - k,
	\end{eqnarray*}
	where $\deg u$ is defined in Definition \ref{degree}, and the normalized graph Laplacian $\Delta_{0}$ is defined as 
	\begin{eqnarray*}
	\Delta_{0} f (u)= \cfrac{1}{\deg u} \sum_{(v, u) \in E} ( f(u) - f(v) )
	\end{eqnarray*}
	for any vertex $u \in V$ and for a function $f: V \to \mathbb{R}$.
\end{thm}
In 2009, Bauer et al. \cite{Jo0} obtained an estimate of the eigenvalues of the normalized graph Laplacian on neighborhood graphs using the Ricci curvature. For simplicial complexes, we \cite{Yamada3} obtained an estimate of the Laplacian using the Ricci curvature. However, we considered only one-dimensional simplicial complexes and assumed many conditions, so the setting of our main result in this paper is more general, and the result is an improvement over our previous result \cite{Yamada3}. This statement is as follows:

\begin{thm}\label{thm:estimate}
Let $K$ be an orientable $(i+1)$-dimensional simplicial complex and $\lambda$ be the eigenvalue of $\Delta^\mathrm{up}_i$. Assume that $\kappa (F, F') \geq k$ for any $F \sim F'$ and for a real number $k$. Then, we have
\begin{eqnarray*}
(i+1)k - i \leq \lambda \leq (i+2) - (i+1) k.
\end{eqnarray*}
\end{thm}
Because a graph is a 1-dimensional simplicial complex, this result expands Theorem \ref{Ol} to a simplicial complex. 

In this paper, we refer to the definition of the Laplacian on simplicial complexes and define the Ricci curvature on simplicial complexes (\S2). In \S3, we prove the upper and lower bounds of the Ricci curvature. In \S4, we prove Theorem \ref{thm:estimate}. In \S5, we present some examples of our results. Finally, we summarize this paper and describe a future plan (\S6).

\section{Preliminaries}
\subsection{Laplacian on simplicial complexes}
 In this section, we present several definitions of simplicial complexes, including the Laplacian and the Ricci curvature of simplicial complexes. An {\em abstract simplicial complex} $K$ is a collection of subsets of a finite set $V$ that is closed under inclusion. An {\em $i$-face} of $K$ is an element of cardinality $i+1$, and $S_{i}(K)$ denotes the set of all $i$-faces of $K$. A face $F$ is said to be {\em oriented} if we choose an ordering on its vertices, and an oriented face $F$ is denoted by $[F]$. The faces that are maximal under inclusion are called {\em facets}, and a simplicial complex $K$ is said to be {\em pure} if all facets have the same dimensions. For any $i$-face $F$, the {\em dimension} of $F$ is $i$, and the dimension of $K$ is the maximum dimension of the faces in $K$. The $i$-th chain group $C_{i}(K, \mathbb{R})$ of $K$ with coefficients in $\mathbb{R}$ is a vector space over the real field $\mathbb{R}$ with basis $B_{i}(K, \mathbb{R})=\left\{ [F] \mid F \in S_{i}(K) \right\}$, and the $i$-th co-chain group $C^{i}(K, \mathbb{R})$ is defined as the dual of the $i$-th chain group.
	\begin{defn}
	For the cochain groups, the {\em simplicial coboundary maps} \\ $\delta_{i} : C^{i}(K, \mathbb{R}) \to C^{i+1}(K, \mathbb{R})$, $i \geq -1$, are defined by
	\begin{eqnarray*}
(\delta_{i} f)([v_{0}, \cdots, v_{i + 1}]) = \sum_{j = 0}^{i + 1} (- 1)^{j} f([v_{0}, \cdots, \hat{v_{j}}, \cdots, v_{i + 1}]).
	\end{eqnarray*}
	for $f \in C^{i}(K, \mathbb{R})$, where $\hat{v_{j}}$ implies that vertex $v_{j}$ has been removed.
	\end{defn}
Note that the one-dimensional vector space $C^{-1}(K, \mathbb{R})$ is generated by a function $f_{-1}$ with $f_{-1} (\emptyset)=1$.
To define the Laplacian, we define the boundary between the oriented face and the inner product.
	\begin{defn}
	Let $[F']=[v_{0}, \cdots, v_{i+1}]$ be an oriented $(i+1)$-face of $K$ and\\ $[F_j] = [v_{0}, \cdots, \hat{v_{j}}, \cdots, v_{i+1}]$ be the oriented $i$-face of $F'$. The {\em boundary of the oriented face} $[F']$ is defined as
		\begin{eqnarray*}
 		\partial [F'] = \sum_{j=0}^{i+1} (-1)^{j}[F_j],
		\end{eqnarray*}
	and the sign $(-1)^{j}$ of $[F_j]$ at the boundary of $[F']$ is denoted by sgn($[F_j], \partial [F']$).
	\end{defn}
	\begin{defn}
	\label{innerdef}
	An {\em inner product} on space $C^{i}(K, \mathbb{R})$ is defined as follows:
		\begin{eqnarray}
		\label{inner}
		(f, g)_{C^{i}} = \sum_{F \in S_{i}(K)} w(F) f([F]) g([F])
		\end{eqnarray}
	for $f$, $g \in C^{i}(K, \mathbb{R})$, where $w: \bigcup_{i = 0} S_{i}(K) \to \mathbb{R}^{+}$ is a function with $w(\emptyset) = 0$. We call $w$ the {\em weight function} on $K$.
	\end{defn}
For the inner product on $C^{i}(K, \mathbb{R})$, the {\em adjoint operator} $\delta^{*}_{i} : C^{i + 1}(K, \mathbb{R}) \to C^{i}(K, \mathbb{R})$ of the coboundary operator $\delta_{i}$ is defined as
	\begin{eqnarray*}
	(\delta_{i} f_{1}, f_{2})_{C^{i+1}} = ( f_{1},\delta^{*}_{i} f_{2})_{C^{i}}
	\end{eqnarray*}
for $f_{1} \in C^{i}(K, \mathbb{R})$ and $f_{2} \in C^{i + 1}(K, \mathbb{R})$, respectively.
	\begin{defn}[Horak-Jost \cite{Ho}]
	\label{i-laplacian}
	We define the Laplacian on $C^{i}(K, \mathbb{R})$ as follows.
		\begin{enumerate}
		\renewcommand{\labelenumi}{{\rm (\arabic{enumi})}}

		\item The $i$-{\em dimensional combinatorial up Laplacian} or simply the {\em $i$-up Laplacian} is defined by
 			\begin{eqnarray*}
  			\mathcal{L}^{\mathrm{up}}_{i}(K) := \delta^{*}_{i} \delta_{i}.
 			\end{eqnarray*}
		\item The {\em $i$-dimensional combinatorial down Laplacian} or simply the {\em $i$-down Laplacian} is defined by
 			\begin{eqnarray*}
  			\mathcal{L}^{\mathrm{down}}_{i}(K) := \delta_{i-1} \delta^{*}_{i-1}.
 			\end{eqnarray*}
		\item The {\em $i$-dimensional combinatorial Laplacian} or simply the {\em $i$-Laplacian} is defined by
 			\begin{eqnarray*}
  			\mathcal{L}_{i}(K) := \mathcal{L}^{\mathrm{up}}_{i}(K) + \mathcal{L}^{\mathrm{down}}_{i}(K) =  \delta^{*}_{i} \delta_{i} + \delta_{i-1} \delta^{*}_{i-1}.
 			\end{eqnarray*}
		\end{enumerate}
	\end{defn}
	For any function $f$ and any $i$-face $F$, the $i$-up Laplacian and the $i$-down Laplacian are given by
	\begin{eqnarray*}
	(\mathcal{L}^{\mathrm{up}}_{i}f)([F]) = \sum_{\bar{F} \in S_{i+1} : F \in \partial \bar{F}} \left\{ \cfrac{w(\bar{F})}{w(F)}f([F]) + \sum_{\substack{F \not\neq F' \in S_{i}: \\ F, F'  \in \partial \bar{F}} } \mathrm{sgn}([F], \partial [\bar{F}]) \mathrm{sgn}([F'], \partial [\bar{F}]) \cfrac{w(\bar{F})}{w(F)} f([F']) \right\}
	\end{eqnarray*}
	and	
	\begin{eqnarray*}
	(\mathcal{L}^{\mathrm{down}}_{i}f)([F]) = \sum_{E \in \partial F} \left\{ \cfrac{w(F)}{w(E)}f([F]) + \sum_{\substack{F' \in S_{i}: \\E= F \cap F'}} \mathrm{sgn}([E], \partial [F]) \mathrm{sgn}([E], \partial [F']) \cfrac{w(F')}{w(E)} f([F']) \right\}.
	\end{eqnarray*}
	As these operators are all self-adjoint and non-negative, the eigenvalues are real and non-negative (see \cite{Ho}).
	\begin{rem}
	\label{eigenvalue}
	Since $\delta_{i} \delta_{i-1} = 0$ and $\delta^{*}_{i-1} \delta^{*}_{i} = 0$, we have $\im \mathcal{L}^{\mathrm{down}}_{i} \subset \kernel \mathcal{L}^{\mathrm{up}}_{i}$ and $\im \mathcal{L}^{\mathrm{up}}_{i} \subset \kernel \mathcal{L}^{\mathrm{down}}_{i}$, where $\im D$ denotes the image and $\kernel D$ denotes the kernel of an operator $D$. Thus, $\lambda$ is a non-zero eigenvalue of $\mathcal{L}_{i}$ if and only if it is a nonzero eigenvalue of either $\mathcal{L}^{\mathrm{up}}_{i}$ or $\mathcal{L}^{\mathrm{down}}_{i}$.
	\end{rem}
We represent the Laplacians in a matrix form. Let $D_{i}$ be the matrix corresponding to operator $\delta_{i}$ and $W_{i}$ be the diagonal matrix representing their scalar product on $C^{i}$. Then, operators $\mathcal{L}^{\mathrm{up}}_{i}(K)$ and $\mathcal{L}^{\mathrm{down}}_{i}(K)$ are expressed as
	\begin{eqnarray*}
	\mathcal{L}^{\mathrm{up}}_{i}(K) = W_{i}^{-1}D_{i}^{T}W_{i+1}D_{i},\\
	\mathcal{L}^{\mathrm{down}}_{i}(K) = D_{i-1}W_{i-1}^{-1}D_{i-1}^{T}W_{i},
	\end{eqnarray*}
where $A^T$ is the transpose of matrix $A$. Using these matrix forms, $\lambda$ is a non-zero eigenvalue of $\mathcal{L}^{\mathrm{up}}_{i}(G)$ if and only if it is a nonzero eigenvalue of $\mathcal{L}^{\mathrm{down}}_{i+1}(G)$.

\begin{defn}\label{degree}
The degree of an $i$-face $F$ is defined by
\begin{eqnarray*}
\deg F = \sum_{\overline{F} \in S_{i+1}(K): F \in \partial \overline{F}} w (\overline{F}).
\end{eqnarray*}
\end{defn}
	\begin{rem}
	\label{normalize}
	If, for every face $F$ of simplicial complex $K$ that is not a facet, we take $w(F) = \deg F$ and the weights of the facets are equal to 1, then the obtained operators are denoted by $\Delta_i^\mathrm{up}$ and $\Delta_i^\mathrm{down}$. In particular, $\Delta_0^\mathrm{up}$ corresponds to the normalized graph Laplacian $\Delta_0$.
	\end{rem}
	
\subsection{Ricci curvature on simplicial complexes}
Hereafter, we consider an abstract simplicial complex $K$ with more than one dimension. We note that the case of one-dimensional simplicial complexes corresponds to the case of undirected graphs. To consider the Ricci curvature on a simplicial complex, we define the distance of a simplicial complex.

\begin{defn}
	\begin{enumerate}
  	\renewcommand{\labelenumi}{(\arabic{enumi})}
	\item Any two $i$-faces $F$ and $F'$ are {\em connected}, denoted by $F \sim F'$, if there exists an $(i+1)$-face $\bar{F}$ such that $F, F' \in \partial \bar{F}$. We denote the {\em $i$-neighborhood of $F$} and {\em $(i+1)$-neighborhood of $F$} by $\Gamma(F)$ and $\Gamma^\mathrm{up}(F)$, respectively; that is, $\Gamma(F) = \left\{ F' \in S_i (K) \mid F' \sim F \right\}$ and $\Gamma^\mathrm{up}(F) = \left\{ \bar{F} \in S_{i+1} (K) \mid F \in \partial \bar{F} \right\}$.
	\item A {\em path} from face $F$ to face $F'$ is a sequence of connected faces $\left\{F_{i-1} \sim F_{i} \right\}_{i=1}^{n}$, where $F_{0} = F$, $F_{n} = F'$. We call $n$ the {\em length} of the path.
   	\item The {\em distance} $d(F, F')$ between two faces $F$ and $F'$ is given by the length of the shortest path from $F$ to $F'$.
	\item $K$ is said to be {\em $(i+1)$-path connected} if any two $i$-faces can be connected by a path.
   	\end{enumerate}
\end{defn}

Next, we define the 1-Wasserstein distance between probability measures.

\begin{defn}
The 1-Wasserstein distance between any two probability measures $\mu$ and $\nu$ on $V$ is given by
	\begin{eqnarray*}
	W(\mu, \nu) = \inf_{A} \sum_{u, v \in V}A(u, v)d(u, v),
	\end{eqnarray*}
	where $A : V \times V \to [0, 1]$ runs over all maps, satisfying
	\begin{eqnarray}
	\label{coupling}
		\begin{cases}
		\sum_{v \in V}A(u, v) = \mu(u),\\
		\sum_{u \in V}A(u, v) = \nu(v).
		\end{cases}
	\end{eqnarray}
Such a map $A$ is called a {\em coupling} between $\mu$ and $\nu$.
\end{defn}
	\begin{rem}
	There exists a coupling $A$ that attains the $1$-Wasserstein distance (see \cite{Le}, \cite{Vi1}, and \cite{Vi2}), which we call an {\em optimal coupling}.
	\end{rem}
	
   One of the most important properties of the $1$-Wasserstein distance is the Kantorovich--Rubinstein duality.
	\begin{prop}[Kantorovich, Rubinstein]\label{kantoro}
The $1$-Wasserstein distance between any two probability measures $\mu$ and $\nu$ on $V$ is written as
		\begin{eqnarray*}
		W(\mu, \nu) = \sup_{f} \sum_{u \in V} f(u)(\mu(u) - \nu(u)),
		\end{eqnarray*}
	where the supremum is taken over all functions $f$ on $V$ that satisfy $|f(u) - f(v)| \leq d(u, v)$ for any $u, v \in V$. A function $f$ on $V$ is said to be {\em $1$-Lipschitz} if $|f(u)-f(v)| \leq d(u,v)$ for any $u, v \in V$.
	\end{prop}

In this paper, we use the following probability measure on $S_i (K)$.
\begin{defn}
For any $i$-face $F \in S_i (K)$, we define a probability measure $m_{F}$ on $S_i (K)$ as
	\begin{eqnarray*}
    	m_{F}(F') =
    		\begin{cases}
    		\cfrac{w(\bar{F})}{(i + 1) \deg F}, &\text{if there exists $\bar{F} \in S_{i+1}(K)$ such that $F, F' \in \partial \bar{F}$}, \\
    		0, & \mathrm{\mathrm{otherwise}}.
    		\end{cases}
   	\end{eqnarray*}
\end{defn}
We note that a probability measure on $S_0 (K)$ corresponds to a simple random walk.
  
Using this probability measure, we define the Ricci curvature.
\begin{defn}\label{Ricci}
For any two distinct $i$-faces $F$ and $F'$, the {\em Ricci curvature} of $F$ and $F'$ is defined as follows:
   	\begin{eqnarray*}
    	\kappa(F, F') =  1 - \cfrac{W(m_{F},m_{F'})}{d(F, F')}.
   	\end{eqnarray*}
\end{defn}

Clearly, we prove the following lemma, and thus, we obtain an upper bound of the Ricci curvature.
\begin{lem}\label{bounded}
For any two distinct $i$-faces $F$ and $F'$, we have
\begin{eqnarray*}
\kappa(F, F') \leq \cfrac{2}{d (F, F')}.
\end{eqnarray*}
\end{lem}
\begin{proof}
We define a delta function $\delta_F$. By the triangle inequality, we have
\begin{eqnarray*}
W(m_{F}, m_{F'}) &\geq& - W(\delta_F, m_{F}),  + W(\delta_F, \delta_F') - W(m_{F'}, \delta_F')\\
&=& - 1 + d(F, F')  - 1,
\end{eqnarray*}
which implies that
\begin{eqnarray*}
d(F,F') - W(m_{F}, m_{F'}) \leq 2.
\end{eqnarray*}
This completes the proof.
\end{proof}

 \section{Properties of the Ricci curvature on simplicial complexes}
 
 \subsection{Upper and lower bounds of the Ricci curvature}
In the case of graphs, Jost-Liu proved a relation between the Ricci curvature and the local clustering coefficient \cite{Jo2}. In this subsection, we extend these properties to the case of simplicial complexes. To do so, we must prepare some notations (see Figure \ref{fig:nbd}).
\begin{eqnarray*}
\Gamma (F, F') &:=& \Gamma (F) \cap \Gamma (F')\\
\Gamma^{\mathrm{up}}(F, \Gamma(F, F')) &:=& \left\{\bar{F} \in \Gamma^{\mathrm{up}}(F) \mid \text{there exists $F'' \in \Gamma(F, F')$ such that $F, F'' \in \partial \bar{F}$} \right\}\\
\Gamma (F, \Gamma(F, F')) &:=& \left\{F'' \in \Gamma(F) \setminus  \Gamma (F, F') \mid \text{there exists $\bar{F} \in \Gamma^{\mathrm{up}}(F, \Gamma(F, F'))$ such that $F'' \in \partial \bar{F}$} \right\}
\end{eqnarray*}

\begin{figure}[h]
  \begin{center}
          \includegraphics[scale=0.4]{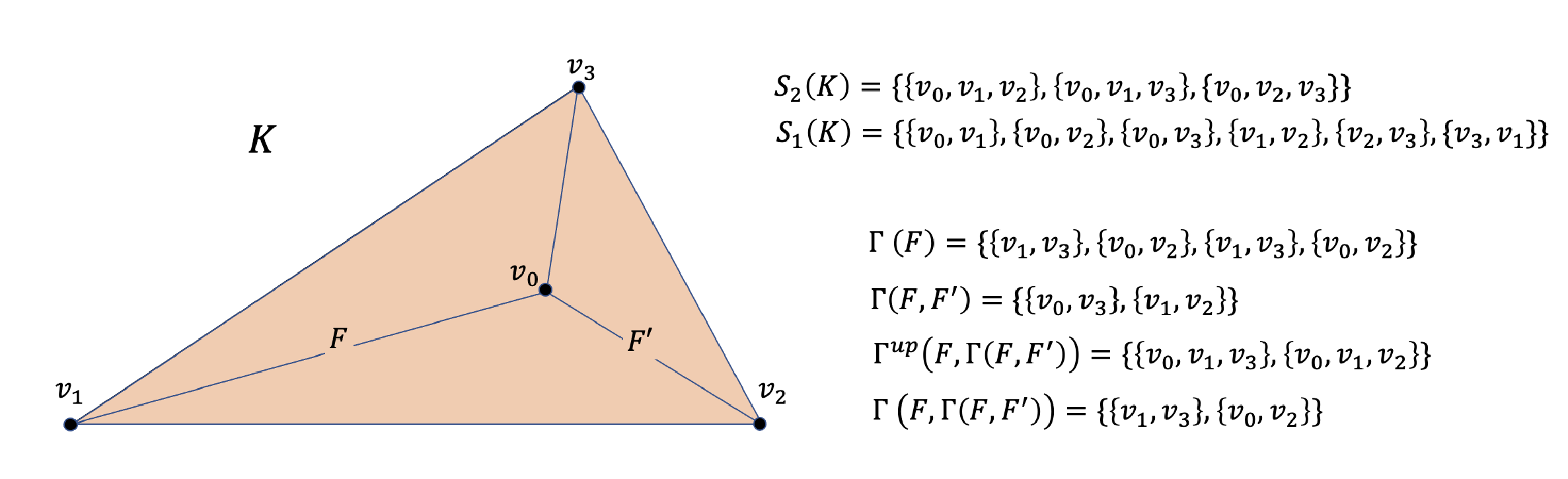}
    \caption{Examples of neighborhoods on a simplicial complex}
    \label{fig:nbd}
  \end{center}
\end{figure}

Before we prove Theorem \ref{lower}, we present the statement again:\\
{\bf Theorem \ref{lower}} For $i$-faces $F$ and $F'$ with $F \sim F'$, if the weights of the facets are equal to 1, then we have
\begin{eqnarray*}
\kappa (F, F') &\geq& - \cfrac{1}{i+1} \left( 1 - \cfrac{1}{\deg F} - \cfrac{1}{\deg F'} - \cfrac{i \# \Gamma (F, F')}{\deg F \vee \deg F'} - \cfrac{\# \Gamma (F, F')}{\deg F \wedge \deg F'}  \right)_+ \\
&\ & -\cfrac{1}{i+1} \left( 1 - \cfrac{1}{\deg F} - \cfrac{1}{\deg F'} - \cfrac{i \# \Gamma (F, F')}{\deg F \vee \deg F'} - \cfrac{\# \Gamma (F, F')}{\deg F \vee \deg F'}  \right)_+\\
&\ & +\cfrac{\# \Gamma (F, F')}{(i+1)(\deg F \vee \deg F')}
\end{eqnarray*}
where $\deg F \wedge \deg F = \min \left\{ \deg F, \deg F' \right\}$, $\deg F \vee \deg F = \max \left\{ \deg F, \deg F' \right\}$, and for $r \in \mathbb{R}$, $(r)_+ = r$ if $r>0$; otherwise, it is 0.
\begin{rem}
For $i = 0$, this statement corresponds to Theorem 3 in \cite{Jo2}. In a general simplicial complex, we need to consider the elements of $\Gamma (F, \Gamma(F, F'))$, which complicates the proof.
\end{rem}
In the proof, we use the following notations:\\
For a real number $R \in \mathbb{R}$ and set $S$, $Al (R, S) := R/|S|$.
\begin{proof}[Proof of Theorem \ref{lower}.]
Without loss of generality, we suppose that
\begin{eqnarray*}
\deg F \wedge \deg F = \deg F,\ \deg F \vee \deg F = \deg F'.
\end{eqnarray*}
We consider a coupling between $m_{F}$ and $m_{F'}$. For simplicity, we define 
\begin{eqnarray*}
\Gamma^{\mathrm{own}} (F) &=& \Gamma (F) \setminus (\Gamma (F, \Gamma(F, F')) \cup \Gamma (F, F') \cup \left\{ F' \right\}),\\
\Gamma^{\mathrm{own}} (F') &=& \Gamma (F') \setminus (\Gamma (F', \Gamma(F, F')) \cup \Gamma (F, F') \cup \left\{ F \right\}).
\end{eqnarray*}
To refer the transfer plan of faces, we define
\begin{eqnarray*}
A := 1 - \cfrac{1}{\deg F} - \cfrac{1}{\deg F'} - \cfrac{i \# \Gamma (F, F')}{\deg F'} - \cfrac{\# \Gamma (F, F')}{\deg F},\\
B:= 1 - \cfrac{1}{\deg F} - \cfrac{1}{\deg F'} - \cfrac{i \# \Gamma (F, F')}{\deg F'} - \cfrac{\# \Gamma (F, F')}{\deg F'}.
\end{eqnarray*}
Subsequently, we obtain $A \leq B$. In addition, because $|\Gamma (F, \Gamma(F, F'))| = |\Gamma (F', \Gamma(F, F'))|$, a bijective map $\phi$ exists between $\Gamma (F, \Gamma(F, F'))$ and $\Gamma (F', \Gamma(F, F'))$.

{\bf Case 1.}  $0 \leq A \leq B$.

Our transfer plan for moving $m_{F}$ to $m_{F'}$ should be as follows:
	\begin{enumerate}
	\item Move the mass of $1/((i+1)\deg F)$ from $F'$ to $i$-faces in $\Gamma^{\mathrm{own}} (F')$. This is well defined because we have $0 \leq B$.
	\item Move the mass of $1 / ((i+1)\deg F')$ to itself at $i$-face in $\Gamma (F, F')$, and move the rest $1 / ((i+1)\deg F) - 1 / ((i+1)\deg F')$ from an $i$-face in $\Gamma (F, F')$ to $i$-faces in $\Gamma^{\mathrm{own}} (F')$.
	\item Move the mass of $1 / ((i+1)\deg F')$ from a face $F \in \Gamma (F, \Gamma(F, F'))$ to $\phi (F) \in \Gamma (F', \Gamma(F, F'))$.
	\item Move the mass of $1/((i+1)\deg F')$ at $i$-faces in $\Gamma^{\mathrm{own}} (F)$, and the rest at $i$-faces in $\Gamma (F, \Gamma(F, F'))$ to $F$. This is well defined because we have 
	\begin{eqnarray*}
	\cfrac{1}{\deg F'} \leq 1 - \cfrac{1}{\deg F} - \cfrac{\# \Gamma (F, F')}{\deg F} - \cfrac{i \# \Gamma (F, F')}{\deg F} + \left( \cfrac{i \# \Gamma (F, F')}{\deg F} - \cfrac{i \# \Gamma (F, F')}{\deg F'} \right),
	\end{eqnarray*}
	which implies $0 \leq A$.
	\item The rest of (4) is
	\begin{eqnarray*}
	1 - \cfrac{1}{\deg F} - \cfrac{1}{\deg F'} - \cfrac{i \# \Gamma (F, F')}{\deg F} - \cfrac{\# \Gamma (F, F')}{\deg F}.
	\end{eqnarray*}
	This is performed in $\Gamma (F, \Gamma(F, F'))$ and $\Gamma^{\mathrm{own}} (F')$ to $i$-faces in $\Gamma^{\mathrm{own}} (F')$.
	\end{enumerate}
Then, by using this transfer plan and calculating the 1-Wasserstein distance between $m_{F}$ and $m_{F'}$, we obtain
\begin{eqnarray*}
(i + 1)W(m_{F}, m_{F'}) &\leq& \cfrac{1}{\deg F} \times 1 + \left( \cfrac{1}{\deg F}  - \ frac {1}{\deg F'}  \right) \times \# \Gamma (F, F') \times 2 \hspace{\fill} \textit{... (1) and (2)}, \\
&\ & + \cfrac{1}{\deg F'} \times i \# \Gamma (F, F') \times 1 \hspace{\fill} + \cfrac{1}{ \deg F'} \times 1 \textit{... (3) and (4)}\\
&\ &   + \left( 1 - \cfrac{1}{\deg F} - \cfrac{1}{\deg F'} - \cfrac{i \# \Gamma (F, F')}{\deg F} - \cfrac{\# \Gamma (F, F')}{\deg F} \right) \times 3 \hspace{\fill} \textit{... (5)}\\
&=& \left( 1 - \cfrac{1}{\deg F} - \cfrac{1}{\deg F'} - \cfrac{i \# \Gamma (F, F')}{\deg F'} - \cfrac{\# \Gamma (F, F')}{\deg F}  \right) \\
&\ & + \left( 1 - \cfrac{1}{\deg F} - \cfrac{1}{\deg F'} - \cfrac{i \# \Gamma (F, F')}{\deg F'} - \cfrac{\# \Gamma (F, F')}{\deg F'}  \right)\\
&\ & -\cfrac{\# \Gamma (F, F')}{\deg F'} + 1
\end{eqnarray*}
which implies that
\begin{eqnarray*}
\kappa (F, F') &\geq& - \cfrac{1}{i+1} \left( 1 - \cfrac{1}{\deg F} - \cfrac{1}{\deg F'} - \cfrac{i \# \Gamma (F, F')}{\deg F'} - \cfrac{\# \Gamma (F, F')}{\deg F}  \right) \\
&\ & -\cfrac{1}{i+1} \left( 1 - \cfrac{1}{\deg F} - \cfrac{1}{\deg F'} - \cfrac{i \# \Gamma (F, F')}{\deg F'} - \cfrac{\# \Gamma (F, F')}{\deg F'}  \right)\\
&\ & +\cfrac{\# \Gamma (F, F')}{(i+1)\deg F'}
\end{eqnarray*}

{\bf Case 2.} $A \leq 0 \leq B$.

Our transfer plan for moving $m_{F}$ to $m_{F'}$ should be as follows:
	\begin{enumerate}
	\item Move the mass of $1/((i+1)\deg F)$ from $F'$ to $i$-faces in $\Gamma^{\mathrm{own}} (F')$. This is well defined because we have $0 \leq B$.
	\item Move the mass of $1 / ((i+1)\deg F')$ to itself at $i$-face in $\Gamma (F, F')$.
	\item Move the mass of $1 / ((i+1)\deg F')$ from a face $F \in \Gamma (F, \Gamma(F, F'))$ to $\phi (F) \in \Gamma (F', \Gamma(F, F'))$.
	\item Move all mass of $1/((i+1)\deg F')$ at $i$-faces in $\Gamma^{\mathrm{own}} (F)$, and the rest at $i$-faces in $\Gamma (F, \Gamma(F, F'))$ to $F$. This is well defined because we have 
	\begin{eqnarray*}
	\cfrac{1}{\deg F'} \geq 1 - \cfrac{1}{\deg F} - \cfrac{\# \Gamma (F, F')}{\deg F} - \cfrac{i \# \Gamma (F, F')}{\deg F} + \left( \cfrac{i \# \Gamma (F, F')}{\deg F} - \cfrac{i \# \Gamma (F, F')}{\deg F'} \right),
	\end{eqnarray*}
	which implies $A \leq 0$.
	\item Move the rest $1 / ((i+1)\deg F) - 1 / ((i+1)\deg F')$ from an $i$-face in $\Gamma (F, F')$ to $i$-faces in $\Gamma^{\mathrm{own}} (F')$.
	\end{enumerate}
Then, by using this transfer plan and calculating the 1-Wasserstein distance between $m_{F}$ and $m_{F'}$, we obtain
\begin{eqnarray*}
(i + 1)W(m_{F}, m_{F'}) &\leq& +  \cfrac{1}{ \deg F'} \times 1 \hfill \textit{... (1)}\\
&\ & + \cfrac{1}{\deg F'} \times i \# \Gamma (F, F') \times 1  \hfill \textit{...(3)}\\\\
&\ & + \cfrac{1}{ \deg F'} \times 1  \hfill \textit{...(4) and (5)}\\
&\ & + \left( 1 - \cfrac{1}{\deg F} - \cfrac{1}{\deg F'} - \cfrac{i \# \Gamma (F, F')}{\deg F'} - \cfrac{\# \Gamma (F, F')}{\deg F'} \right) \times 2  \hfill \textit{...(4) and (5)}\\
&=& \left( 1 - \cfrac{1}{\deg F} - \cfrac{1}{\deg F'} - \cfrac{i \# \Gamma (F, F')}{\deg F'} - \cfrac{\# \Gamma (F, F')}{\deg F'}  \right) -\cfrac{\# \Gamma (F, F')}{\deg F'} + 1
\end{eqnarray*}
which implies that
\begin{eqnarray*}
\kappa (F, F') &\geq& -\cfrac{1}{i+1} \left( 1 - \cfrac{1}{\deg F} - \cfrac{1}{\deg F'} - \cfrac{i \# \Gamma (F, F')}{\deg F'} - \cfrac{\# \Gamma (F, F')}{\deg F'}  \right) +\cfrac{\# \Gamma (F, F')}{(i+1)\deg F'}
\end{eqnarray*}

{\bf Case 3.} $A \leq B \leq 0$.

Our transfer plan for moving $m_{F}$ to $m_{F'}$ should be as follows:
	\begin{enumerate}
	\item Move the mass of $1 / ((i+1)\deg F')$ to itself at $i$-face in $\Gamma (F, F')$, and move the rest to $F$.
	\item Move $1/((i+1)\deg F')$ from $F'$ to $i$-faces in $\Gamma^{\mathrm{own}} (F')$. This is well defined because we have
	\begin{eqnarray*}
	 \cfrac{1}{\deg F} \geq 1 - \cfrac{1}{\deg F'} - \cfrac{i \# \Gamma (F, F')}{\deg F'} - \cfrac{\# \Gamma (F, F')}{\deg F'}
	\end{eqnarray*}
	which implies $B \leq 0$. Then, move the rest to $F$.
	\item Move all mass of $1/((i+1)\deg F')$ at $i$-faces in $\Gamma^{\mathrm{own}} (F)$, and the rest at $i$-faces in $\Gamma (F, \Gamma(F, F'))$ to $F$. This is well defined because we have 
	\begin{eqnarray*}
	\cfrac{1}{\deg F'} \geq 1 - \cfrac{1}{\deg F} - \cfrac{\# \Gamma (F, F')}{\deg F} - \cfrac{i \# \Gamma (F, F')}{\deg F} + \left( \cfrac{i \# \Gamma (F, F')}{\deg F} - \cfrac{i \# \Gamma (F, F')}{\deg F'} \right),
	\end{eqnarray*}
	which implies $A \leq 0$.
	\item Move the mass of $1 / ((i+1)\deg F')$ from a face $F \in \Gamma (F, \Gamma(F, F'))$ to $\phi (F) \in \Gamma (F', \Gamma(F, F'))$, and move the rest to $F$.
	\end{enumerate}
Then, by using this transfer plan and calculating the 1-Wasserstein distance between $m_{F}$ and $m_{F'}$, we have
\begin{eqnarray*}
(i + 1)W(m_{F}, m_{F'}) &\leq& \left( 1 - \cfrac{\# \Gamma (F, F')}{(i+1)\deg F'}  \right) \times 1
\end{eqnarray*}
which implies that
\begin{eqnarray*}
\kappa (F, F') &\geq& \cfrac{\# \Gamma (F, F')}{(i+1)\deg F'}
\end{eqnarray*}

Then, this completes the proof.
\end{proof}

\begin{thm}\label{upper}
Let $F$ and $F'$ be $i$faces with $F \sim F'$. We have
\begin{eqnarray*}
\kappa (F, F') \leq \cfrac{\# \Gamma (F, F')}{(i+1) (\deg F \vee \deg F')}
\end{eqnarray*}
\end{thm}
\begin{rem}
For $i = 0$, this statement corresponds to Theorem 4 in \cite{Jo2}. Moreover, this theorem is proved in the same manner as Theorem 4 in \cite{Jo2}.
\end{rem}

Regarding the lower bound of the Ricci curvature, the following lemma implies that considering the Ricci curvature of connected faces is sufficient, although the Ricci curvature is defined for any pair of faces.
\begin{prop}[Lin-Lu-Yau \cite{Yau1}]\label{LLY1}
If $\kappa(F, F') \geq k$ for any connected faces $F \sim F'$ and for a real number $k$, then $\kappa(F, F') \geq k$ for any pair of faces $(F, F') \in S_i \times S_i$.
\end{prop}
This proposition is proved in the same manner as Lemma 2.3 in \cite{Yau1}.

\section{Estimate of the eigenvalues of the Laplacian by the Ricci curvature}
In this section, we discuss the relation between the eigenvalue of the normalized $i$-up Laplacian and the Ricci curvature. 
To prove Theorem \ref{thm:estimate}, we have stated a few definitions and theorems.
\begin{defn}
Let $K$ be an $(i+1)$-path connected simplicial complex. 
\begin{enumerate}
\item $K$ is {\em orientable} if an orientation exists on the $(i+1)$-faces of $K$, such that for any $(i+1)$-faces $F$ satisfying $E, E' \in \partial F$, where $E$ and $E'$ are $i$-faces, $\mathrm{sgn}([E], \partial [F]) \mathrm{sgn}([E'], \partial [F]) = 1$ holds.
\item $K$ is {\em opposite orientable} if an orientation exists on the $(i+1)$-faces of $K$, such that for any $(i+1)$-faces $F$ satisfying $E, E' \in \partial F$, where $E$ and $E'$ are $i$-faces, $\mathrm{sgn}([E], \partial [F]) \mathrm{sgn}([E'], \partial [F]) = -1$ holds.
\end{enumerate}
\end{defn}

\begin{rem}
For graphs, a constant function is an eigenfunction of the eigenvalue $0$; therefore, for $i \geq 1$, we consider the eigenvalues of the normalized $i$-up Laplacian. If $K$ is orientable, then a constant function is an eigenfunction of the eigenvalue $i+2$. In contrast, if $K$ is opposite orientable, no eigenvalue exists such that the eigenfunction is a constant function.
\end{rem}

Before we prove Theorem \ref{thm:estimate}, we present the statement again:\\
{\bf Theorem \ref{thm:estimate}} Let $K$ be an orientable $(i+1)$-path-connected simplicial complex and $\lambda$ be the eigenvalue of $\Delta^\mathrm{up}_i$, except for $i+2$. Assume that $\kappa (F, F') \geq k$ for any $F \sim F'$ and for a real number $k$. Then, we have
\begin{eqnarray*}
(i+1)k - i \leq \lambda \leq (i+2) - (i+1) k.
\end{eqnarray*}

\begin{proof}[Proof of Theorem \ref{thm:estimate}]
Let $f$ be an eigenfunction with respect to $\lambda$. Based on this assumption, $f$ is not a constant function; hence, by scaling $f$, if necessary,
	\begin{eqnarray}\label{lip}
	\sup_{F, F' \in S_i (K)}\cfrac{|f([F]) - f([F'])|}{d(F, F')} = 1.
	\end{eqnarray}
We fix any two $i$faces $F, F' \in S_i(K)$, and we have $(1 - \kappa (F, F'))d(F,F') = W(m_{F}, m_{F'})$. Based in the definition of the probability measure on $S_i$, $\Delta^{\mathrm{up}}_{i}$ is represented as follows:
	\begin{eqnarray}\label{measure}
	(\Delta^{\mathrm{up}}_{i} f) (F)= f([F]) + (i+1)\sum_{F'' \in S_i} m_F ([F'']) f(F'') = \lambda f([F])
	\end{eqnarray}
	Then, using the Kantorovich duality and Eq. \eqref{measure}, we calculate the Wasserstein distance.
	\begin{eqnarray*}
	|(1 - \kappa (F, F')) d(F,F')|  &=& W(m_{F}, m_{F'})\\
	&\geq& | \sum_{F'' \in S_i (K)} f([F'']) (m_{F}^{\epsilon}(F'')- m_{F'}^{\epsilon}(F''))|, \\
	&=& \left| \cfrac{\lambda - 1}{i+1} \right| |f([F]) - f([F'])|
	\end{eqnarray*}
Thus, based on the assumption and Eq. \eqref{lip}, we obtain
\begin{eqnarray*}
1 - k \geq \cfrac{1}{i+1}| \lambda - 1|
\end{eqnarray*}
which implies that
\begin{eqnarray*}
(i+1)k - i \leq \lambda \leq (i+2) - (i+1) k.
\end{eqnarray*}
\end{proof}

\begin{cor}
\label{cor:estimate}
Let $K$ be an opposite orientable $(i+1)$-path-connected simplicial complex and $\lambda$ be the eigenvalue of $\Delta^\mathrm{up}_i$. Assume that $\kappa (F, F') \geq k$ for any $F \sim F'$ and for a real number $k$. Then, we have
\begin{eqnarray*}
(i+1)k - i \leq \lambda \leq (i+2) - (i+1) k.
\end{eqnarray*}
\end{cor}
$\Delta^{\mathrm{up}}_{i}$ is represented as follows:
	\begin{eqnarray}\label{measure}
	(\Delta^{\mathrm{up}}_{i} f) (F)= f([F]) - (i+1)\sum_{F'' \in S_i} m_F ([F'']) f(F'') = \lambda f([F])
	\end{eqnarray}
The proof is the same as Theorem \ref{thm:estimate}.
\begin{rem}
Because any connected graph is opposite orientable, the eigenvalue of the normalized graph Laplacian is
\begin{eqnarray*}
k \leq \lambda \leq 2-k.
\end{eqnarray*}
Thus, Corollary \ref{cor:estimate} is a generalization of Theorem \ref{Ol}. 
\end{rem}

\section{Example}
\begin{ex}
\label{ex:tetra}
If we consider an $(i+1)$-face $K$ (for $i=2$, see Figure \ref{fig:example}), then for $F, F' \in S_i (K)$, the probability measure is
\begin{eqnarray*}
m_F (F') =
\begin{cases}
\cfrac{1}{i+1},& \textit{if $F \sim F'$},\\
0,& \textit{otherwise}.
\end{cases}
\end{eqnarray*}
and the Ricci curvature is
\begin{eqnarray*}
\kappa (F, F') = \cfrac{i}{i+1},\  \text{for any}\ F \sim F'.
\end{eqnarray*}
Thus, by Theorem \ref{thm:estimate}, we obtain
\begin{eqnarray*}
0 \leq \lambda \leq 2.
\end{eqnarray*}
In contrast, in \cite{Ho}, the eigenvalues are $0$ and $i+2$; thus, the value corresponds to the lower bound of our result.
\end{ex}

\begin{ex}
\label{ex2}
We consider $K_1 = \left\{v_0, v_1, v_2, v_3 \right\}$ with $S_2(K_1) = \left\{ \left\{v_0, v_1, v_2 \right\}, \left\{v_1, v_2, v_3 \right\} \right\}$. (see Figure \ref{fig:example}). From Theorem \ref{lower} and Theorem \ref{upper}, for any $F, F' \in S_1 (K_1) \setminus \left\{ \left\{v_1, v_2 \right\} \right\}$, we obtain
\begin{eqnarray*}
\kappa (\left\{v_1, v_2 \right\}, F) &=& \cfrac{1}{4},\\
\kappa (F, F') &=& \cfrac{1}{2}.
\end{eqnarray*}

In contrast, $K_1$ is orientable because we orient $2$-faces and $1$-faces as follows:
\begin{eqnarray*}
[v_0, v_1, v_2],\ [v_1, v_2, v_3],
\end{eqnarray*}
\begin{eqnarray*}
[v_0, v_1],\ [v_1, v_2],\ [v_2, v_0],\ [v_2, v_3],\ [v_3, v_1].
\end{eqnarray*}
Thus, by Theorem \ref{thm:estimate}, we have
\begin{eqnarray*}
- \cfrac{1}{2} \leq \lambda \leq \cfrac{5}{2}.
\end{eqnarray*}
\end{ex}

\begin{ex}
We consider $K_2 = \left\{v_0, v_1, v_2, v_3, v_4, v_5 \right\}$ with 
\begin{eqnarray*}
S_2 (K_2) =\left\{
 \begin{tabular}{l}
 $\left\{v_0, v_1, v_2 \right\}, \left\{v_0, v_2, v_3 \right\}, \left\{v_0, v_3, v_4 \right\}, \left\{v_0, v_4, v_1 \right\}$,\\
  $\left\{v_5, v_1, v_2 \right\}, \left\{v_5, v_2, v_3 \right\}, \left\{v_5, v_3, v_4 \right\}, \left\{v_5, v_4, v_1 \right\}$
  \end{tabular}
  \right\}.
\end{eqnarray*}
From Theorem \ref{lower} and Theorem \ref{upper}, for any $F, F' \in S_1 (K_2)$, we obtain
\begin{eqnarray*}
\kappa (F, F') &=& \cfrac{1}{4}.
\end{eqnarray*}

In contrast, $K_2$ is orientable because we orient $2$-faces and $1$-faces, similar to Example \ref{ex2}. Thus, by Theorem \ref{thm:estimate}, we have
\begin{eqnarray*}
- \cfrac{1}{2} \leq \lambda \leq \cfrac{5}{2}.
\end{eqnarray*}
\end{ex}
\begin{figure}[h]
        \begin{center}
          \includegraphics[scale=0.4]{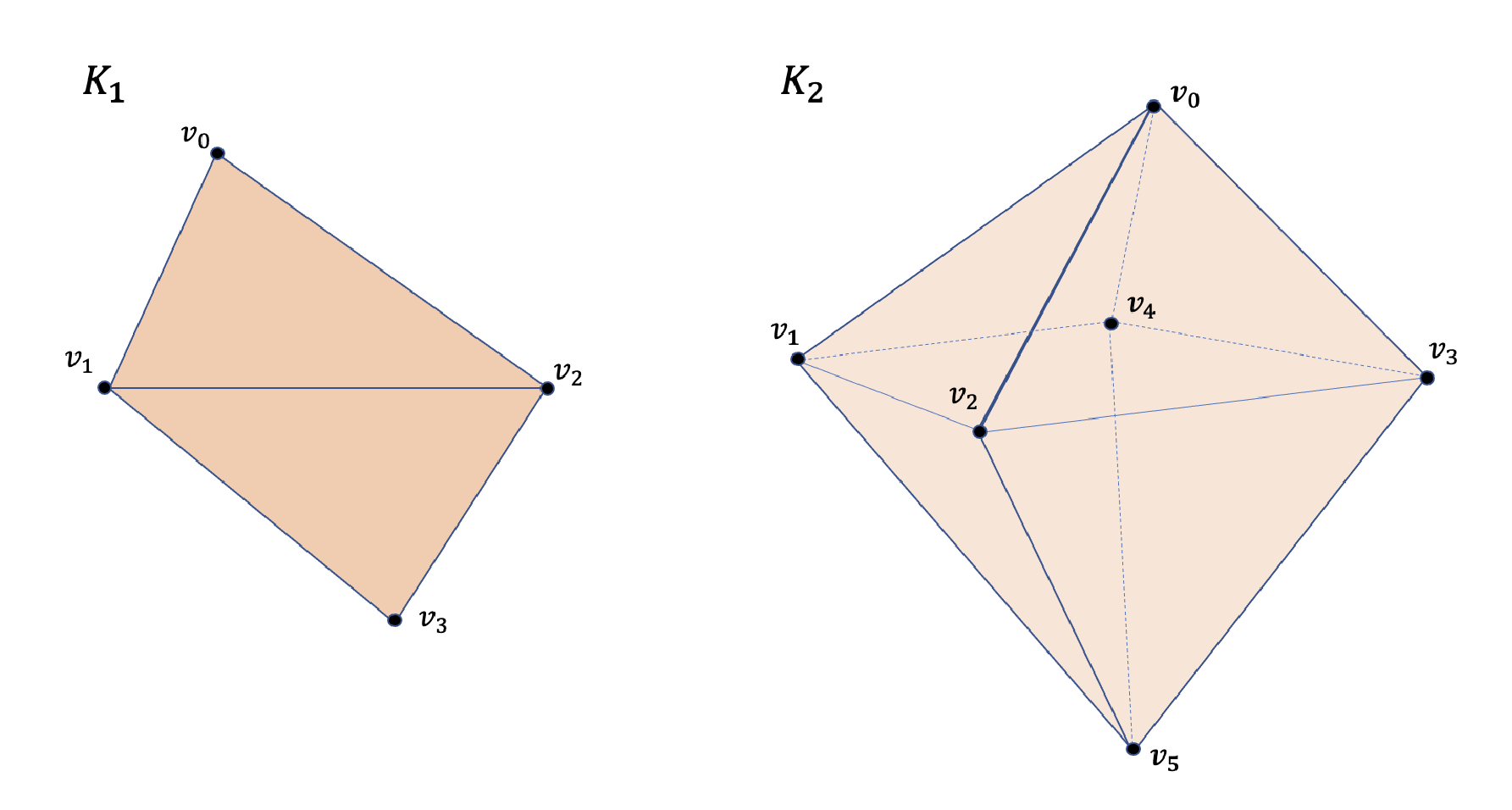}
\caption{Example of simplicial complexes}
    \label{fig:example}
  \end{center}
\end{figure}

\newpage

\section{Conclusion}
\begin{enumerate}
\item Theorem \ref{lower} and Theorem \ref{upper} are generalizations of the result in \cite{Jo2}. Based on these results, if we consider the Ricci curvature on simplicial complexes, the calculation is more complicated than that implemented using the graphs.
\item A primary result, namely Theorem \ref{thm:estimate}, is the most general estimate of the non-zero eigenvalue of $i$-Laplacian. However, this estimate may not be statistically significant.
\item To prove Theorem \ref{thm:estimate}, we assumed the orientation of faces. This differs from the undirected case; therefore, removing these conditions is essential. To do this, a different approach is required.
\item If the Ricci curvature is negative, our results have no meaning, and thus, we intend to improve our results.
\end{enumerate}


\end{document}